\documentclass[12pt,oneside]{amsart}
\usepackage{amssymb,amscd,amsthm,verbatim}
\usepackage[left=1in,right=1in,top=1in,bottom=1in]{geometry}
\usepackage{latexsym}
\usepackage{color}
\usepackage{xcolor}
\usepackage{graphicx}
\usepackage{mathrsfs}
\usepackage[all,cmtip]{xy}
\usepackage{bbm}
\usepackage{stmaryrd}
\usepackage{xr}

\usepackage[
	colorlinks,
	linkcolor={black},
	citecolor={black},
	urlcolor={black}
]{hyperref}

\date{\today}

\newcommand{\N}{{\mathbb{N}}}

\newcommand{\PP}{{\mathbb P}}

\newcommand{\R}{{\mathbb R}}

\newcommand{\T}{{\mathbb T}}

\newcommand{\X}{{\mathbb S}}

\newcommand{\Z}{{\mathbb Z}}

\newcommand{\ED}{{\mathcal{ED}}}

\newcommand{\rot}{{\mathrm{rot}}}

\renewcommand{\ker}{{\mathrm{Ker}}}

\newtheorem{theorem}{Theorem}[section]
\newtheorem{lemma}[theorem]{Lemma}
\newtheorem{prop}[theorem]{Proposition}

\theoremstyle{definition}
\newtheorem{claim}{Claim}[theorem]

\newtheorem{remark}[theorem]{Remark}

\theoremstyle{plain}

\newenvironment{claimproof}[1][Proof of Claim]{\noindent \underline{#1.} }{\hfill$\diamondsuit$}

\allowdisplaybreaks 

\numberwithin{equation}{section}

\DeclareMathOperator{\supp}{supp}

\newcommand{\set}[1]{\left\{#1\right\}}

\definecolor{purple}{rgb}{.5,0,1}

\begin{document}

\title[The Essential Spectrum of the Doubling Map Model is Connected]{The Almost Sure Essential Spectrum of the\\Doubling Map Model is Connected}

\author{David Damanik}
\address{Department of Mathematics, Rice University, Houston, TX~77005, USA}
\email{damanik@rice.edu}
\thanks{D.D.\ was supported in part by NSF grants DMS--1700131 and DMS--2054752 and Simons Fellowship $\# 669836$}

\author{Jake Fillman}
\address{Department of Mathematics, Texas State University, San Marcos, TX 78666, USA}
\email{fillman@txstate.edu}
\thanks{J.F.\ was supported in part by Simons Foundation Collaboration Grant \#711663.}

\maketitle

\begin{abstract}
We consider discrete Schr\"odinger operators on the half line with potentials generated by the doubling map and continuous sampling functions. We show that the essential spectrum of these operators is always connected. This result is obtained by computing the subgroup of the range of the Schwartzman homomorphism associated with homotopy classes of continuous maps on the suspension of the standard solenoid that factor through the suspension of the doubling map and then showing that this subgroup characterizes the topological structure of the spectrum.
\end{abstract}


\hypersetup{
	linkcolor={black!30!blue},
	citecolor={red},
	urlcolor={black!30!blue}
}

\section{Introduction}

The doubling map model is the discrete half-line Schr\"odinger operator 
\begin{equation}\label{e.dmoper}
[H_\omega \psi](n) = \psi(n+1) + \psi(n-1) + f(2^n \omega) \psi(n), \quad n \geq 0
\end{equation}
in $\ell^2(\Z_+)$, where $\omega \in \T = \R/\Z$, $f \in C(\T,\R)$, and $\psi \in \ell^2(\Z_+)$, with the boundary condition $\psi(-1)=0$. This family has been studied by Chulaevsky--Spencer \cite{ChulaevskySpencer1995}, Bourgain--Schlag \cite{BourgainSchlag2000CMP}, Damanik--Killip \cite{DamKil2005b}, Zhang \cite{Zhang2016}, Bjerkl\"ov \cite{Bjerkloev2020}, and Avila-Damanik-Zhang \cite{AviDamZha2020}, among others. 

These authors were primarily interested in the spectral type of this operator, but to make the main result of \cite{BourgainSchlag2000CMP} meaningful, Bourgain and Schlag had to prove the following statement about the spectrum of this operator: for Lebesgue almost all $\omega \in \T$, the spectrum $\sigma(H_\omega)$ contains the interval $[-2+f(0),2+f(0)]$. We are interested in a related result concerning the global topological structure of the spectrum.

The doubling map model is a prominent example of an ergodic family of Schr\"odinger operators. Such a family is generated by an ergodic measurable dynamical system $(\Omega,T,\mu)$ and a measurable and (for simplicity) bounded $f : \Omega \to \R$. One generates potentials via
\begin{equation}\label{e.pot}
V_\omega(n) = f(T^n \omega)
\end{equation}
and bounded self-adjoint operators via
\begin{equation}\label{e.oper}
[H_\omega \psi](n) = \psi(n+1) + \psi(n-1) + V_\omega(n) \psi(n).
\end{equation}
The doubling map model arises upon choosing $\Omega = \T$, $T \omega = 2 \omega$, and $\mu = \mathrm{Leb}$.

Two general remarks are in order. If the transformation $T$ is invertible, then \eqref{e.pot} can be defined for any $n \in \Z$ and then the operators \eqref{e.oper} are usually considered in $\ell^2(\Z)$. In the case of the doubling map, the non-invertibility of $T$ suggests that we initially define the operators in $\ell^2(\Z_+)$. However, as we will see later, it is still desirable to pass to an associated whole line model. This strategy was also put to use by Damanik and Killip in \cite{DamKil2005b}. The second remark is that in cases where $\Omega$ is a compact metric space and $T$ is continuous, one often restricts attention to continuous sampling functions $f : \Omega \to \R$. The reasoning underlying both remarks the same: the general theory is nicer in the invertible case (e.g., the spectrum is $\mu$-almost sure constant and the discrete spectrum is $\mu$-almost surely empty and several important aspects of Kotani theory need the presence of two half lines) and for continuous sampling function in the topological setting (e.g., the spectrum is globally independent of $\omega$ if $T$ is minimal).

Returning to the case of the doubling map, given these two remarks, we are especially interested in continuous $f$ and will restrict attention to those. Moreover, the best we can say is that the essential spectrum of $H_\omega$ is Lebesgue almost surely independent of $\omega$, while the discrete spectrum may be present and depend sensitively on $\omega$. We will denote the almost sure essential spectrum by
\begin{equation}\label{e.dmasespec}
\Sigma_f = \sigma_\mathrm{ess}(H_\omega) \quad \mu-\text{a.s.}
\end{equation}

Here is the main result of this note:

\begin{theorem}\label{t.dmspec}
For every $f \in C(\T,\R)$, the almost sure essential spectrum of the doubling map model $\Sigma_f$ is connected.
\end{theorem}

\begin{remark}\label{rema:1.2}
Let us make a few comments about Theorem~\ref{t.dmspec}
\begin{enumerate}
\item[(a)] The theorem shows that $\Sigma_f$ is always an interval. As pointed out above, the argument given in  \cite{BourgainSchlag2000CMP} shows that this interval must contain $[-2+f(0),2+f(0)]$. The point is that $0$ is a fixed point of $T$ and hence the spectrum of $H_0$ is easy to compute: it is given precisely by $[-2+f(0),2+f(0)]$. More generally, one can see that for every periodic point of $T$ the associated periodic Schr\"odinger spectrum must be contained in $\Sigma_f$. Pushing this further, one can view $\Sigma_f$ as the closure of the union of all periodic spectra arising in this way. As there are infinitely many periodic spectra to deal with, and a point of minimal period $p$ generically gives a spectrum with $p-1$ gaps, the fact that $\Sigma_f$ has no gaps whatsoever is not obvious.
\smallskip

\item[(b)] The conclusion of the theorem may fail for discontinuous sampling functions $f$. Indeed, while there still is an almost sure essential spectrum $\Sigma_f$ for any measurable (and, say, bounded) $f : \T \to \R$, it may be disconnected. An explicit example is given by $f(\omega) = \lambda \chi_{[0,1/2)}(\omega)$ with $\lambda > 4$; in this case one has a half-line Bernoulli-Anderson model and a well-known argument shows that $\Sigma_f = [-2,2] \cup [-2 + \lambda,2 + \lambda]$. In the spirit of a more general conjecture of Bellissard, we suspect that for any bounded measurable $f$, $\Sigma_f$ has at most finitely many gaps.
\smallskip

\item[(c)] Bourgain and Schlag also considered Schr\"odinger operators in $\ell^2(\Z)$ with potentials generated by hyperbolic toral automorphisms in \cite{BourgainSchlag2000CMP}, again with a focus on the spectral type. The result analogous to Theorem~\ref{t.dmspec} (i.e., that the almost sure spectrum is always connected) holds for these operators as well. This was discussed in \cite{DFGap}.
\smallskip

\item[(d)] The proof of Theorem~\ref{t.dmspec} can be applied to any affine toral endomorphism of the form $\T^d \ni \omega \mapsto A\omega$ for which Lebesgue measure is ergodic, and hence the result of this note generalizes the discussion from the previous item to the \emph{non-invertible} case. We focus on the doubling map as the most interesting special case which nevertheless exhibits the relevant phenomena and challenges. Concretely, if $A \in \Z^{n \times n}$ has $\det(A) \neq 0$, then $A$ induces a measure-preserving endomorphism $T:\T^d \to \T^d$ (this is well known, compare the discussion in \cite[Section~1.1]{Walters1982:ErgTh}). If $A$ has no roots of unity as eigenvalues, then Lebesgue measure on $\T^d$ is $T$-ergodic \cite[Corollary~1.10.1]{Walters1982:ErgTh} and can then be lifted to a suitable ergodic measure on the natural extension \cite{Ruelle2004} (cf.~the discussion on pp.912--913 of \cite{Mihailescu2011}). \smallskip

\item[(e)] One crucial motivation for us to prove Theorem~\ref{t.dmspec} is provided by spectral pseudo-randomness. Heuristically speaking, a model is \emph{pseudo-random} if it has spectral properties akin to those of random models. In the context of discrete one-dimensional Schr\"odinger operators this would mean almost sure pure point spectrum with exponentially decaying eigenfunctions (a.k.a.\ spectral localization), perhaps along with a suitable dynamical localization statement, and the finiteness of the number of gaps of the almost sure essential spectrum. Note that localization properties were the focus of the Bourgain-Schlag paper \cite{BourgainSchlag2000CMP}, whereas the topological structure of the almost sure essential spectrum had not been discussed before (to the best of our knowledge). In order to study the transition into the pseudo-random realm, it is a good idea to consider some examples. For simplicity, let us consider the following four types of potentials:
\begin{align}
V_\omega^{\mathrm{(qp)}} & = 2 \lambda \cos (2 \pi (n \alpha + \omega)) \\
V_\omega^{\mathrm{(ss)}} & = 2 \lambda \cos (2 \pi (n^2 \alpha + \omega)) \\
V_\omega^{\mathrm{(gss)}} & = 2 \lambda \cos (2 \pi (n^k \alpha + \omega)) \\
V_\omega^{\mathrm{(dm)}} & = 2 \lambda \cos (2 \pi (2^n \omega)).
\end{align}
The associated Schr\"odinger operators will be denoted accordingly, that is, $H_\omega^{\mathrm{(qp)}}, H_\omega^{\mathrm{(ss)}}, H_\omega^{\mathrm{(gss)}}, H_\omega^{\mathrm{(dm)}}$. Here, the superscripts stand for \emph{quasi-periodic, skew-shift, generalized skew-shift}, and \emph{doubling map}, respectively. The number $\alpha$ is assumed to be irrational. The reader will recognize $H_\omega^{\mathrm{(qp)}}$ as the \emph{almost Mathieu operator} and it is chosen here as a representative quasi-periodic model for definiteness. Now, $H_\omega^{\mathrm{(qp)}}$ is not pseudo-random. Indeed, it has purely absolutely continuous spectrum for each $\lambda \in (-1,1)$, as shown by Avila \cite{A14a}, and it has a Cantor spectrum for every $\lambda \in \R \setminus \{ 0 \}$, a result due to Avila and Jitomirskaya \cite{AviJit2009Annals}. Spectral localization is conjectured to hold for each of $H_\omega^{\mathrm{(ss)}}, H_\omega^{\mathrm{(gss)}}H_\omega^{\mathrm{(dm)}}$, for any $\lambda \in \R \setminus \{ 0 \}$. Also, the (almost sure) spectrum of $H_\omega^{\mathrm{(ss)}}, H_\omega^{\mathrm{(gss)}}H_\omega^{\mathrm{(dm)}}$ is conjectured to be an interval. There are only partial results in support of these conjectures \cite{Bourgain2002, BourgainGoldsteinSchlag2001, HanLemmSchlag2020ETDS, HanLemmSchlag2020JST, Kruger2009, Kruger2012JSP, Kruger2012JFA, Kruger2013}. Clearly, $H_\omega^{\mathrm{(ss)}}$ is the least random of these models, which are conjectured to be pseudo-random. The main reason to also consider the generalized skew-shift model $H_\omega^{\mathrm{(gss)}}$ is to make the analysis a bit easier, though it is still expected to be very difficult. In light of this discussion, Theorem~\ref{t.dmspec} establishes one of the pseudo-randomness aspects for one of the key candidates in full generality.
\end{enumerate}
\end{remark}

The proof of Theorem~\ref{t.dmspec} is given in Section~\ref{sec:mainproof}. The overall strategy is to relate the doubling map to the standard solenoid, which then gives an invertible dynamical system to which the gap-labelling theorem can be applied. One then computes the Schwartzman homomorphism restricted to the homotopy classes of maps on the suspension of the solenoid that factor through the suspension of the doubling map, and then shows that any stable section of uniformly hyperbolic cocycles associated with the doubling map factor through in this manner. By showing that this group is precisely $\Z$, one sees that the only possible rotation numbers that one can observe in spectral gaps are zero and one, and hence there are no open interior gaps.

\section*{Acknowledgments}

We are grateful to Vaughn Climenhaga and Anton Gorodetski for helpful conversations. We also want to thank the American Institute of Mathematics for hospitality and support during a January 2022 visit, during which part of this work was completed.

\section{Absence of Gaps via Embedding and Schwartzman} \label{sec:mainproof}

In this section we associate a family of whole-line operators with the half-line family generated by the doubling map $T : \T \to \T$, along with the ergodic measure $\mu = \mathrm{Leb}$ and the sampling function $f \in C(\T,\R)$ we are interested in. The motivation for doing so is that we want to apply general theorems that are known for ergodic families of Schr\"odinger operators in $\ell^2(\Z)$. Specifically, we seek to invoke the theorem that establishes the existence of an almost sure spectrum for the whole line family and mimic the arguments that lead to a canonical set of gap labels for this set.

In principle we want to proceed as in \cite{DamKil2005b} when passing from the half-line model to the whole-line model. However, there is one aspect that will force us to proceed differently from  \cite{DamKil2005b}. Recall first that \cite{DamKil2005b} used the binary expansion of $\omega \in \T$ to semi-conjugate $T$ to a one-sided full shift over the alphabet $\{0,1\}$. The latter dynamical system has an obvious two-sided extension: the two-sided full shift over the alphabet $\{0,1\}$. With a simple adjustment of the ``forward-looking'' conjugated sampling function, we can easily extend it from $\{0,1\}^{\Z_+}$ to $\{0,1\}^\Z$. This setup is sufficient to identify the almost sure spectrum of the derived whole-line model with the almost sure essential spectrum of the half-line model, and hence we could view $\Sigma_f$ from this perspective. However, when working out the consequences of the gap labelling theorem, the total disconnectedness of $\{0,1\}^\Z$ presents a serious obstacle when trying to prove that $\Sigma_f$ is connected! Our solution will be to not pass to the symbolic setting, but rather use the standard solenoid to make the doubling map invertible. 

\subsection{The Associated Whole-Line Model}

Let us recall the construction of the standard (Smale--Williams) solenoid; compare \cite[Section~1.9]{BrinStuck2015Book} and \cite[Section~17.1]{KatokHassel1995Book}. Consider the solid torus
$$
\mathfrak{T} = \T \times D^2, \quad \text{where } D^2 = \{ (x,y) \in \R^2 : x^2 + y^2 \le 1 \}.
$$
Fix $\lambda \in (0,1/2)$ and define
$$
F : \mathfrak{T} \to \mathfrak{T}, \quad (\omega,x,y) \mapsto \left( 2 \omega, \lambda x + \frac12 \cos (2 \pi \omega), \lambda y + \frac12 \sin (2 \pi \omega) \right). 
$$
Then $F$ is one-to-one and
$$
S = \bigcap_{n = 0}^\infty F^n(\mathfrak{T})
$$
is a closed $F$-invariant subset of $\mathfrak{T}$ on which $F$ is a homeomorphism; $S$ is called the (standard) \emph{solenoid}.

We can now define the desired family of whole-line operators. We set
$\widetilde \Omega = S$, $\widetilde T = F|_S$, and choose $\widetilde \mu$ to be the natural ergodic extension to $\widetilde \Omega$ of Lebesgue measure on $\T$. The measure $\widetilde \mu$, besides being ergodic, has full topological support,
\begin{equation}\label{eq:tildemusupport}
\supp \widetilde \mu = \widetilde \Omega,
\end{equation}
and can be interpreted as the Bowen--Margulis measure, as well as the Sinai--Ruelle--Bowen measure. Moreover, it is locally the direct product of Lebesgue measure on $\T$ with the $(1/2,1/2)$-Bernoulli measure along the Cantor fibers; in particular, with $\pi_1(\omega,x,y) = \omega$, we have
\begin{equation}\label{e.bowenprojection}
(\pi_1)_*(\widetilde \mu) = \mu.
\end{equation}

While we were unfortunately not able to locate an explicit discussion of this measure in the literature due to the simplicity of the example of the standard solenoid within the discussion of  hyperbolic attractors, and the Gibbs measures thereupon, we refer the reader to \cite{ClimenhagaLuzzattoPesin2017, ClimenhagaPesinZelerowicz2019, PPGibbs, Ruelle2004} as well as \cite[Exercise 2.1.9]{EinsiedlerWard2011Book} and \cite[Appendix to Chapter 1]{Robert2000} for useful background information.

As the sampling function, we choose
\begin{equation}\label{e.tildefdef}
\widetilde f : \widetilde\Omega \to \R, \; (\omega,x,y) \mapsto f(\omega).
\end{equation}
This in turn yields potentials
\begin{equation}\label{e.tildepot}
\widetilde V_{(\omega,x,y)}(n) = \widetilde f(\widetilde T^n (\omega,x,y)), \quad n \in \Z
\end{equation}
and bounded self-adjoint operators
\begin{equation}\label{e.tildeoper}
[H_{(\omega,x,y)} \psi](n) = \psi(n+1) + \psi(n-1) + \widetilde V_{(\omega,x,y)}(n) \psi(n), \quad \psi \in \ell^2(\Z), \quad n \in \Z.
\end{equation}
The associated \emph{density of states measure} $dk$ is given by
\begin{equation} \label{eq:dsm}
\int g \, dk = \int \langle \delta_0, g(H_{(\omega,x,y)}) \delta_0 \rangle \, d\widetilde \mu(\omega,x,y),
\end{equation}
and its accumulation function
\begin{equation} \label{eq:ids}
k(E) = \int \chi_{(-\infty,E]} \, dk
\end{equation}
is called the \emph{integrated density of states}.

Since $\widetilde \mu$ is $\widetilde T$-ergodic, the general theory of ergodic Schr\"odinger operators in $\ell^2(\Z)$ (see, e.g., \cite{ESO1}) gives that there is a compact $\widetilde \Sigma_{\widetilde f} \subseteq \R$ such that
\begin{equation}\label{e.dmaswlspec}
\widetilde \Sigma_{\widetilde f} = \sigma(H_{(\omega,x,y)}) \quad \widetilde \mu-\text{a.s.}
\end{equation}
Moreover, this almost sure spectrum coincides with the topological support of the density of states measure, that is,
\begin{equation} \label{eq:sigmasuppdk}
\widetilde \Sigma_{\widetilde f} = \supp  dk.
\end{equation}

Our goal is to show that the set $\widetilde \Sigma_{\widetilde f}$ coincides with the set of interest, $\Sigma_f$. Before we can prove this, we need to establish the following:

\begin{lemma}\label{l.potentialscoincide}
For $(\omega,x,y) \in S$ and $n \in \Z_+$, we have
$$
\widetilde V_{(\omega,x,y)}(n) = V_\omega(n).
$$
\end{lemma}

\begin{proof}
This follows quickly from the definitions:
$$
\widetilde V_{(\omega,x,y)}(n) = \widetilde f(\widetilde T^n (\omega,x,y)) = \widetilde f(F^n (\omega,x,y)) = f(T^n \omega) = V_\omega(n).
$$
Here we used in the third step that for $n \in \Z_+$, the first component of $F^n (\omega,x,y)$ is simply $T^n \omega$, so that \eqref{e.tildefdef} yields the asserted identity.
\end{proof}

\begin{prop}\label{p.spectracoincide}
We have $\widetilde \Sigma_{\widetilde f} = \Sigma_f$.
\end{prop}

\begin{proof}
It follows from the general theory of ergodic Schr\"odinger operators in $\ell^2(\Z)$  that for $\widetilde \mu$-almost every $(\omega,x,y) \in S$, the spectrum of $H_{(\omega,x,y)}$ is purely essential, equals $\widetilde \Sigma_{\widetilde f}$, and coincides with the essential spectrum of the restriction of $H_{(\omega,x,y)}$ to $\Z_+$ (this follows for instance from the discussion in \cite[Sections~2.2 and~4.2]{ESO1}). The latter set only depends on $\omega$ and hence, by almost sure independence, Lemma~\ref{l.potentialscoincide}, and \eqref{e.bowenprojection}, will coincide with $\Sigma_f$.
\end{proof}

By  \eqref{eq:tildemusupport} and \cite[Theorem 3.1]{Johnson1986JDE} (see also \cite[Theorem~3.8.2 and Corollary~4.9.4]{ESO1}) we have
\begin{equation} \label{eq:JohnsonForSolenoid}
\widetilde \Sigma_{\widetilde f} = \R \setminus \ED,
\end{equation}
where 
$$
\ED = \{ E \in \R : (\widetilde T, A_{E-\widetilde f}) \text{ enjoys an exponential dichotomy} \}.
$$
Here,
$$
A_{E-\widetilde f} : \widetilde \Omega \to \mathrm{SL}(2,\R), \quad (\omega,x,y) \mapsto \begin{bmatrix} E - \widetilde f(\omega,x,y) & - 1 \\ 1 & 0 \end{bmatrix}
$$
and $(\widetilde T, A_{E-\widetilde f})$ is the associated $\mathrm{SL}(2,\R)$-cocycle over the base dynamics given by $\widetilde T$,
$$
(\widetilde T, A_{E-\widetilde f}) : \widetilde \Omega \times \R^2 \to \widetilde \Omega \times \R^2, \quad (\omega,x,y,v) \mapsto A_{E-\widetilde f} (\omega,x,y) v.
$$
For $n \in \Z$, the maps $A_{E-\widetilde f}^n : \widetilde \Omega \to \mathrm{SL}(2,\R)$ are defined by $(\widetilde T, A_{E-\widetilde f})^n = (\widetilde T^n, A^n_{E-\widetilde f})$.

For $E \in \R \setminus \widetilde\Sigma_{\widetilde f} = \ED$, the definition of exponential dichotomy asserts that there exist  continuous maps $\widetilde\Lambda_E^\pm:\widetilde\Omega \to \R\PP^1$ and constants $C,c>0$ such that
\begin{align} \label{eq:expdich1}
A_{E-\widetilde f}(\widetilde\omega) \widetilde \Lambda_E^\pm(\widetilde\omega) 
& = \widetilde \Lambda_E^\pm(\widetilde T\widetilde \omega), \quad  \widetilde\omega \in \widetilde\Omega, \\
\label{eq:expdich2}
\|A_{E - \widetilde f}^{\pm n}(\widetilde\omega) v\| 
& \leq Ce^{-cn}, \quad v \in \widetilde\Lambda_E^{\pm}(\widetilde\omega), \ n \in \N.
\end{align}
We call $\widetilde\Lambda_E^+$ (respectively, $\widetilde\Lambda_E^-$) the \emph{stable section} (respectively, the \emph{unstable section}) of $(\widetilde T,A_{E-\widetilde f})$ . For the topological arguments in the present work, let us point out that we only use the invariance property, not the exponential decay statements for semiorbits.

Of course we have a similar family $A_{E-f}^n : \Omega \to \mathrm{SL}(2,\R)$ defined by
$$
A_{E-f} : \Omega \to \mathrm{SL}(2,\R), \quad \omega \mapsto \begin{bmatrix} E - f(\omega) & - 1 \\ 1 & 0 \end{bmatrix}
$$
and $(T, A_{E-f})^n = (T^n, A^n_{E-f})$ for $n \in \Z_+$.
\begin{lemma}\label{l.cocyclescoincide}
For $(\omega,x,y) \in S$ and $n \in \Z_+$, we have
$$
A^n_{E-\widetilde f}(\omega,x,y) = A_{E-f}^n(\omega).
$$
\end{lemma}
\begin{proof} This is a consequence of Lemma~\ref{l.potentialscoincide}.\end{proof}

\subsection{Suspensions of the Doubling map and the Standard Solenoid}

Let us briefly recall some terminology and definitions that will be helpful. For proofs and further discussion, we point the reader to the survey \cite{DFGap}. We consider a compact metric space $X$ with a continuous flow $\tau$ and a $\tau$-ergodic measure, $\nu$. Let $C^\sharp(X,\T)$ denote the set of homotopy classes of continuous maps $X \to \T$. Given $\phi \in C(X,\T)$, $x \in X$, one can lift the map $\phi_x:t \mapsto \phi(\tau^t x)$ to $\psi_x : \R \to \R$. The limit
\[\rot(\phi;x) = \lim_{t \to \infty} \frac{\psi_x(t)}{t}\]
exists for $\nu$-a.e.\ $x$, it is almost-surely independent of $x$, and its almost-sure value depends only on the homotopy class of $\phi$ \cite{Schwarzmann1957Annals}. The induced map $\mathfrak{A}_\nu:C^\sharp(X,\T) \to \R$ given by
\[ \mathfrak{A}_\nu([\phi]) = \rot(\phi;x) \quad \nu\text{-a.e.\ } x \in X, \]
is called the \emph{Schwartzman homomorphism}. When working with linear cocycles over a dynamical system, is often convenient to work with maps into the projective line $\R\PP^1$ instead of $\T$. For such maps, one can define $\mathfrak{A}_\nu$ by identifying $\R\PP^1$ with $\T$ via the map $\T \ni \theta \mapsto \mathrm{span}\{(\cos\pi\theta,\sin\pi\theta)^\top\} \in \R\PP^1$. Using this identification, if $\Lambda \in C(X,\R\PP^1)$, one has
\begin{equation}
\mathfrak{A}_\nu([\Lambda]) 
= \lim_{T\to \infty} \frac{1}{\pi T} \Delta_{\rm arg}^{[0,T]} \Lambda(\tau^t x), \quad \nu\text{-a.e.\ } x \in X,
\end{equation}
where $\Delta_{\rm arg}^I$ denotes the net change in the argument on the interval $I$.

The \emph{gap-labelling theorem} for ergodic Schr\"odinger operators (see, e.g.,  \cite{ESO1, DFGap, Johnson1986JDE}) asserts the following: if $\{H_\omega\}_{\omega \in \Omega}$ is an ergodic family of Schr\"odinger operators in $\ell^2(\Z)$ generated by an invertible topological dynamical system $(\Omega,T)$ with ergodic measure $\mu$, and $k$ denotes the associated integrated density of states, then for each $E \in \R \setminus \Sigma$, $k(E)$ lies in the range of the Schwartzman homomorphism associated to $(X,\tau,\nu)$, the suspension of $(\Omega,T,\mu)$.  More precisely, recall that $E \in \mathcal{ED}$ on account of  \eqref{eq:JohnsonForSolenoid}, and hence enjoys (un)stable sections $\Lambda_E^\pm$ as in \eqref{eq:expdich1}--\eqref{eq:expdich2}. The gap-labelling theorem asserts that
\begin{equation}\label{idsrn} 
k(E) = 1 - \mathfrak{A}_\nu \Big(  \big[ \overline\Lambda^+_E\big]\Big),
\end{equation}
where $\overline\Lambda^+_E$ is a suitable suspension of the stable section associated with the cocycle $(T,A_{E-f})$, which can be generated by using a suitable homotopy to the identity; see \cite{DFGap, Johnson1986JDE} for details. For an earlier version of the gap-labelling theorem via $K$-theory that applies to operators in arbitrary dimension, see Bellissard \cite{Bel1986, Bel1992b}.

Consider $\Omega = \T$, $T\omega = 2\omega$, and $X = X(\Omega,T)$ the suspension
\[ \T \times [0,1] / ((\omega,1) \sim (2\omega,0)). \]
Let us emphasize that non-invertibility of $(\Omega,T)$ implies that one cannot directly apply the gap-labelling theorem in this setting. Nevertheless, the following calculation of homotopy classes will be useful.

\begin{theorem} \label{t:dmSuspCsharp}
With notation as above, $C^\sharp(X,\T) \cong \Z$, generated by the trivial map $[x,t] \mapsto t$.
\end{theorem}

\begin{proof}
Let $\phi \in C(X,\T)$ be given. For each $s \in [0,1]$, consider the map $\phi_s:\T \to \T$ given by $\omega \mapsto \phi([\omega,s])$. By well-known facts from topology, $\phi_s$ is homotopic to $\omega \mapsto k\omega$ for some $k \in \Z$. Since $\phi_s$ is homotopic to $\phi_{s'}$ for each $s,s'$, the value of $k$ is independent of $s$. 
On the other hand, one has
\[\phi_1(\omega) = \phi([\omega,1])= \phi([2\omega,0]) = \phi_0(2\omega).\]
Thus, $k=2k$, forcing $k = 0$, so $\phi_s$ is nullhomotopic for every $s$. Since $[0,1] = [2 \cdot 0,0] = [0,0]$, the set $\{[0,s] : s \in [0,1]\}$ is homeomorphic to a circle and hence there is $n \in \Z$ such that the restriction of $\phi$ to that circle is homotopic to $s \mapsto ns$. Consider $\phi_0[\omega,s] = \phi[\omega,s]-ns$. By work above, $\phi_0$ is homotopic to a map $\eta_0$ that vanishes on the set
\[ Y = \set{[0,s] \in X : s \in [0,1]}\cup \set{[\omega,0] \in X : \omega \in \T}.\]
Collapsing $Y$ to a point, we see that $\eta_0$ factors through a map $X/Y \to \T$. Since $X/Y$ is equivalent to the space obtained from the unit square $[0,1] \times [0,1]$ by identifying all points on the boundary, $X/Y \cong \X^2$, the two-sphere, and hence it follows that $\eta_0$ is nullhomotopic. Thus, $\phi$ is homotopic to $[x,t] \mapsto nt$, as desired.
\end{proof}

\begin{remark} \label{rem:generalTdendo}
Clearly, the proof of Theorem~\ref{t:dmSuspCsharp} applies to any expanding map $\omega \mapsto m \omega$ with $m \in \{2,3,\ldots\}$. Furthermore, the conclusion also holds for the suspension of any toral endomorphism of the form 
\[\T^d \ni \omega \mapsto A\omega,\]
where $A$ is a $d \times d$ integer matrix for which $\ker(A^*-I)$ is trivial; in particular, this applies to ergodic toral endomorphisms, which have no roots of unity as eigenvalues. This generalizes this discussion in \cite[Section~8]{DFGap} to the non-invertible case.
\end{remark}

Let $(\widetilde\Omega,\widetilde T,\widetilde\mu)$ denote the standard solenoid, let 
\begin{equation}
\widetilde X = X(\widetilde\Omega,\widetilde T)  = \widetilde \Omega \times [0,1] / ((\widetilde\omega,1) \sim (\widetilde T \widetilde\omega,0))
\end{equation}
 be its suspension, and let $\widetilde \nu$ denote the suspension of $\widetilde\mu$. The topologies of $\widetilde\Omega$ and $\widetilde X$ are somewhat more complicated than those of $\Omega = \T$ and $X = X(\Omega,T)$. However, in the case in which a map $\widetilde{X} \to \T$ factors through $X$, one can use the previous result to study the Schwartzman homomorphism. More precisely, notice that
\[ p:\widetilde X \to X, \quad [(\omega,x,y),s] \mapsto [\omega,s] \]
is continuous on account of the calculation
\[p[(\omega,x,y),1] = [\omega,1] = [2\omega,0] = p[\widetilde T(\omega,x,y),0]\]
We say that $\widetilde\phi \in C(\widetilde{X},\T)$ \emph{factors through} $X$ if there is $\phi \in C(X,\T)$ such that $\widetilde\phi = \phi\circ p$, that is,
\begin{equation}
\widetilde\phi([(\omega,x,y),s]) = \phi([\omega,s])
\end{equation}
for all $ (\omega,x,y) \in \widetilde\Omega$ and $s \in [0,1]$.

\begin{theorem} \label{t:solenoidThroughTCsharp}
 If $\widetilde\phi \in C(\widetilde X,\T)$ factors through $X$, then $\mathfrak{A}_{\widetilde \nu}([\widetilde\phi]) \in \Z$.
\end{theorem}

\begin{proof}
Write $\widetilde\phi = \phi \circ p$ for some $\phi \in C(X,\T)$. By Theorem~\ref{t:dmSuspCsharp}, $\phi$ is homotopic to $\chi_n:[\omega,s]\mapsto ns$ for some $n \in \Z$. Writing $F:X \times I \to \T$ for a homotopy from $\phi$ to $\chi_n$, note that
\[\widetilde{F}:([\omega,x,y,s],t) \mapsto F([\omega,s],t)\]
gives a homotopy from $\widetilde \phi$ to the map $\widetilde\chi_n:[\omega,x,y,s]\mapsto ns$, and the result follows by noting that
\[\mathfrak{A}_{\widetilde\nu}([\widetilde\chi_n]) = n\]
by a direct calculation.
\end{proof}

\subsection{Proof of Main Theorem}

We now put everything together to prove the main result.

\begin{proof}[Proof of Theorem~\ref{t.dmspec}]
Let $f \in C(\T,\R)$ be given and write $\Sigma = \Sigma_f = \widetilde{\Sigma}_{\widetilde f}$ (cf.\ Proposition~\ref{p.spectracoincide}).  Given $E \in \R \setminus \Sigma$, the cocycle $(\widetilde T, A_{E-\widetilde f})$ is uniormly hyperbolic by \eqref{eq:JohnsonForSolenoid}, and thus there exist $\widetilde\Lambda^\pm$ as in \eqref{eq:expdich1} and \eqref{eq:expdich2}.

Now, define $A^t_{E - \widetilde f}$ for arbitary $t \in \R$ by using a suitable smooth homotopy to the identity as in \cite{DFGap, Johnson1986JDE}. More precisely,  let $\theta$ and $\lambda$ be smooth nondecreasing functions so that $\theta\equiv 0$ in a neighborhood of $0$, $\theta \equiv \pi/2$ in a neighborhood of $1/2$, $\lambda \equiv 0$ in a neighborhood of $1/2$ and $\lambda \equiv 1$ in a neighborhod of $1$, and then define
\begin{equation} \label{def:homotopy:identity}
Y_{E-\widetilde f}(\widetilde\omega,t)
=
\begin{cases}
\begin{bmatrix}
\cos(\theta(t)) & -\sin(\theta(t)) \\
\sin(\theta(t)) & \cos(\theta(t))
\end{bmatrix}
&
0 \leq t \leq 1/2
\\[5.5mm]
\begin{bmatrix}
\lambda(t)(E - \widetilde f(\widetilde \omega)) & -1 \\
1 & 0
\end{bmatrix}
&
1/2 \leq t \leq 1
\end{cases}
\end{equation}

With this, we define $A_{E- \widetilde f}^t(\widetilde \omega)$ by using $Y_{E- \widetilde f}$ to interpolate between $A_{E- \widetilde f}^n$ and $A_{E- \widetilde  f}^{n+1}$. More precisely, put
\begin{equation}
A_{E- \widetilde f}^t(\widetilde \omega)
=
Y_{E- \widetilde f}\left(\widetilde T^n \widetilde \omega,t-n \right) A_{E- \widetilde f}^{n}(\widetilde \omega),
\quad
\widetilde \omega \in \widetilde \Omega, \; n \le t < n+1,
\end{equation}
where $n \in \Z$. One can check that $A_{E - \widetilde f}^t(\widetilde \omega)$ is a smooth function of $t$ for all fixed $\widetilde \omega \in \widetilde \Omega$ and $E \in \R$ that agrees with $A_{E- \widetilde f}^n$ when restricted to $\Z$. 

Denote the suspension of the solenoid by $\widetilde{X} = X(\widetilde\Omega,\widetilde T)$, and use $A_{E-\widetilde f}^t$ to produce a continuous section $\overline\Lambda^+:\widetilde X \to \R\PP^1$ by
\begin{equation} \label{eq:Lambda+interpolatedDef}
\overline\Lambda^+([\widetilde\omega,s]) = A^s(\widetilde \omega) \widetilde\Lambda^+(\widetilde\omega), \quad \widetilde\omega \in \widetilde\Omega, \ s \in [0,1].
\end{equation}

\begin{claim}
The map $\overline{\Lambda}^+$ from  \eqref{eq:Lambda+interpolatedDef} is well-defined and continuous.
\end{claim}

\begin{claimproof}
By invariance, one has
\[ \widetilde\Lambda^+([\widetilde \omega,1]) 
= A^1(\widetilde\omega)\widetilde\Lambda^+(\widetilde\omega) 
= \widetilde \Lambda^+(\widetilde T \widetilde\omega) 
= \widetilde\Lambda^+([\widetilde T \widetilde\omega,0]), \]
which shows both that $\overline\Lambda^+$ is well-defined and that it is continuous.
\end{claimproof}

\begin{claim}
The map $\widetilde\Lambda^+$ depends only on the first coordinate of $\widetilde \omega$. That is, there exists a continuous map $\Lambda^+:\T \to \R\PP^1$ such that
\begin{equation}
\widetilde\Lambda^+(\omega,x,y) = \Lambda^+(\omega), \quad \forall (\omega,x,y) \in \widetilde\Omega.
\end{equation}
\end{claim}

\begin{claimproof} Consider $\widetilde\omega = (\omega,x,y) \in \widetilde\Omega$. By examining the proof of \cite[Theorem~1.2]{DFLY2016DCDS}, one sees that $\widetilde\Lambda^+(\widetilde\omega)$ is given by the limit of the most contracted direction of $A_{E-\widetilde f}^n(\widetilde\omega)$ as $n \to \infty$. By Lemma~\ref{l.cocyclescoincide}, this is then precisely the limit of the most contracted direction of $A^n_{E-f}(\omega)$ as $n\to \infty$, which then only depends on $\omega$. The claim follows. \end{claimproof}
\medskip

By the claim, $\overline\Lambda^+$ factors through $X$. Thus, $\mathfrak{A}_{\widetilde \nu}(\overline\Lambda^+) \in \Z$ by Theorem~\ref{t:solenoidThroughTCsharp}, which together with \eqref{eq:dsm}--\eqref{eq:ids} and \eqref {idsrn} implies that $\mathfrak{A}_{\widetilde \nu}(\overline\Lambda^+) \in \{0,1\}$. Thus, by \eqref{eq:ids} and \eqref{eq:sigmasuppdk}, we have $E <\min\Sigma$ or $E > \max \Sigma$, so it follows that $\Sigma$ has no interior gaps.
\end{proof}

\bibliographystyle{abbrvArXiv}

\bibliography{gapbib}

\end{document}